\documentclass[10pt,reqno]{amsart}
  \usepackage{geometry}
\usepackage{amsthm}
\usepackage{amsmath}
\usepackage{amssymb}
\usepackage{mathrsfs}
\usepackage{latexsym}
\usepackage{exscale}
\usepackage{geometry}
\usepackage{graphicx}
\usepackage[utf8]{inputenc}
\usepackage{verbatim}

\usepackage{fancyhdr}

\usepackage{color}
\definecolor{red}{rgb}{1,0.1,0.1}
\definecolor{blue}{rgb}{0.1,0.1,1}
\definecolor{vb}{RGB}{160,32,240}

\theoremstyle{plain}

\newtheorem*{teo*}{Theorem}
\newtheorem*{prop*}{Proposition}
\newtheorem*{lema*}{Lemma}

\numberwithin{equation}{section}
\newtheorem{teo}{Theorem}[section]
\newtheorem{lema}[teo]{Lemma}

\newtheorem{prop}[teo]{Proposition}

\theoremstyle{remark}
\newtheorem{obs}[teo]{Remark}

\theoremstyle{definition}

\newtheorem*{mydef*}{Definition}

\allowdisplaybreaks

\newcommand{\R}{\mathbb{R}^n}

\newcommand{\N}{\mathbb{N}}

\newcommand{\A}{\mathcal{A}}

\begin{document}
\title[Fractional operator with $L^{\alpha,s}$-H\"ormander conditions]{Sharp bounds for fractional operator with $L^{\alpha,s}$-H\"ormander conditions}




\author[G.~H.~Iba\~{n}ez~Firnkorn]{Gonzalo H. Iba\~{n}ez-Firnkorn}
\address{G.~H.~Iba\~{n}ez~Firnkorn\\ FaMAF \\ Universidad Nacional de C\'ordoba \\
CIEM (CONICET) \\ 5000 C\'ordoba, Argentina}
\email{gibanez@famaf.unc.edu.ar}

\author[M.~S.~Riveros]{Mar\'{\i}a Silvina Riveros}
\address{M.~S.~Riveros \\ FaMAF \\ Universidad Nacional de C\'ordoba \\
CIEM (CONICET) \\ 5000 C\'ordoba, Argentina}
\email{sriveros@famaf.unc.edu.ar}

\author[R.~E.~Vidal]{Ra\'ul E. Vidal}
\address{R.~E.~Vidal \\ FaMAF \\ Universidad Nacional de C\'ordoba \\
CIEM (CONICET) \\ 5000 C\'ordoba, Argentina}
\email{vidal@famaf.unc.edu.ar}

\thanks{The authors are  partially supported by
CONICET and SECYT-UNC}



\maketitle

\begin{abstract}
In this paper 
 we provide the sharp bound
for a fractional type
o\-pe\-ra\-tor given by a kernel satisfying  the
$L^{\alpha,s}$-H\"ormander condition and fractional size
condition,   $0<\alpha<n$ and $1< s\leq \infty$. To prove this
result we use a new appropriate sparse domination
provided in this work. Examples of these operators 
include the fractional rough operators. For
 the case $s=\infty$ we recover  the sharp
 bound of the fractional integral, $I_{\alpha}$, proved in
[Lacey, M. T., Moen, K., P\'erez, C., Torres, R. H. (2010). Sharp
weighted bounds for fractional integral operators. Journal of
Functional Analysis, 259(5), 1073-1097].
\end{abstract}

\maketitle

\section{Introduction and main results}

Let $0<\alpha<n$, the fractional integral operator $I_{\alpha}$ on
$\R$ is defined by
$$I_{\alpha}f(x):=\int_{\R}\frac{f(y)}{|x-y|^{n-\alpha}}dy.$$
This operator is bounded from $L^p(dx)$ into $L^q(dx)$ provided
that $1<p<\frac{n}{\alpha}$ and
$\frac1{q}=\frac1{p}-\frac{\alpha}{n}$ (see \cite{LibroStein} for
this result).

In the study of weighted estimates for the fractional integral, the class of weights con\-si\-de\-red is the $A_{p,q}$  introduced
 by Muckenhoupt and Wheeden  in \cite{MW74}. Recall that $w$ is a weight if it  is a non-negative locally integrable  function.  Given $1<p<q<\infty$,
the weight $w$ is in the class
  $ A_{p,q}$ if
$$[w]_{A_{p,q}}:=\sup_{Q}\left(\frac1{|Q|}\int_Q w^q\right)\left(\frac1{|Q|}\int_Q w^{-p'}\right)^{q/p'}<\infty.$$
If $w\in A_{p,q}$ then $w^q\in A_{1+q/p'}$ with
$[w^q]_{{1+q/p'}}=[w]_{A_{p,q}}$ and $w^{-p'}\in A_{1+p'/q}$
with $[w^{-p'}]_{{1+p'/q}}=[w]_{A_{p,q}}^{p'/q}$ where $A_s$ denotes
the classical Muckenhoupt class of weights. Observe that $w\in
A_{p,p}$ is equivalent to $w^p\in A_{p}$. The class
$A_{\infty}=\cup_{p\geq 1}A_{p}$, and the statement $w\in
A_{\infty,\infty}$ is equivalent to $w^{-1}\in A_1$.

There have been  several works devoted to the study of quantitative
weighted estimates, in other words, in those papers the authors
study how these estimates depend on the weight constant $[w]_{A_p}$
or $[w]_{A_{p,q}}$.
 The estimate for the Hardy-Littlewood maximal function was studied by Buckley in \cite{B93}. This result attracted renewed attention after the work of Astala, Iwaniec and Saksman \cite{AIS01} on the theory of Quasirregular mapppings. They proved sharp regularity results for solutions to the Beltrami equation, assuming that the operator norm of the Beurling-Ahlfors transform grows linearly in terms of the $A_p$ constant for $p\geq 2$. This linear growth was proved by S. Petermichl and A. Volberg in \cite{PV02}. This result opened up the possibility of considering some other operators.  S. Petermichl  in \cite{P07,P08} proved the corresponding results for the Hilbert transform and the Riesz Transforms.  
 The $A_2$ Theorem,  namely the linear dependence on the $A_2$ constant for
Calder\'on-Zygmund operators, proved by Hyt\"onen in \cite{H12}, can
be considered the most representative in this line. In the case of
fractional integrals, the sharp dependence of the $A_{p,q}$
constants was obtained by Lacey, Moen, P\'erez and Torres
\cite{LMPT10}. The precise statement is the following

\begin{teo}\cite{LMPT10}\label{SharpIalpha}
Let $0<\alpha<n$, $1<p<\frac{n}{\alpha}$ and $\frac1{q}=\frac1{p}-\frac{\alpha}{n}$. If $w\in A_{p,q}$ then
$$\|I_{\alpha}f\|_{L^q(w^q)}\leq c_{n,\alpha}[w]_{A_{p,q}}^{(1-\frac{\alpha}{n})\max\left\{1,\frac{p'}{q}\right\}}\|f\|_{L^p(w^p)},$$
and the estimate is sharp in the sense that the inequality does not hold if we replace the exponent of the $A_{p,q}$ constant by a smaller one.
\end{teo}

The Calder\'on-Zymund integral operators can be generalized  by  taking
other regularity condition of the kernel, for example $L^{r'}$-H\"ormander
condition. This integral operators  are controlled in the $L^p$-norm
sense by the maximal operator $M_r$,  defined by $L^r$-average.
 For
more details see, for example, \cite{LR05}, \cite{MPTG05}. The
operator $I_{\alpha}$ can be generalized in  an analogous way by adding
an assumption of boundedness, as in \cite{K90}, or adding some
fractional size condition, as in \cite{BLR11}.

Now we give the definitions of the fractional size and H\"{o}rmander
conditions. First we introduce some notation, we set
$$\|f\|_{s,B}=
\left( \frac1{|B|}\int_B |f|^s\right)^{1/s},$$  where $B$ is a
ball. Observe that in this averages the balls $B$ can be replaced
by cubes $Q$. The notation  $|x|\sim t$ means $t < |x|\leq 2t$ and
we write
$$\|f\|_{ s,|x|\sim t}=\|f \chi_{|x|\sim
t}\|_{s,B(0,2t)}.$$

Let 
$0< \alpha < n$ and $1\leq s\leq \infty$.
The function $g$ is said to satisfy the fractional size condition, 
 $g\in S_{\alpha,s}$, if there exists a constant $C>0$ such that
$$\|g\|_{s,|x| \sim t} \leq C t^{\alpha - n}.$$

For $s=1$ we write $S_{\alpha,s}=S_{\alpha}$. Observe that if
$K_{\alpha}\in S_{\alpha}$, then there exists a constant $c>0$ such
that
$$\int_{|x|\sim t} |g(x)|dx\leq c t^{\alpha}.$$


The function  $h$ satisfies the $L^{\alpha,s}$-H\"ormander
condition ($h \in H_{\alpha,s}$), if there exist $c_{s}>1$
and $C_{s}>0$ such that for all $x$ and $R>c_{s}|x|$,
 \begin{align*} \sum_{m=1}^{\infty} (2^mR)^{n- \alpha}  \| h(\cdot - x) - h(\cdot)\|_{s,|y|\sim2^mR} \leq C_{s}.\end{align*}
We say that $h \in H_{\alpha,\infty}$ if $h$ satisfies the previous condition with $\|\cdot\|_{L^{\infty},|x|\sim 2^mR}$ in place
 of $\|\cdot\|_{s,|x|\sim 2^mR}$.
For $\alpha =0$ we write $H_{0, s}=H_{s}$, the classical
$L^s$-H\"ormander condition.

Observe  that if  $K_{\alpha}(x)=|x|^{n-\alpha}$ then
 $T_{\alpha}=I_{\alpha}$ the fractional integral and $K_\alpha\in S_{\alpha, \infty}\cap H_{\alpha, \infty}$.

In this paper we consider the following fractional operator. Let
$0<\alpha<n$,  $1\leq r<\infty$ and $r'$  be the conjugated exponent of
$r$. Let $K_{\alpha}$ be a measurable function defined away from
$0$,
such that $K_\alpha\in S_{\alpha,r'}\cap H_{\alpha,r'}$.  For any 
$f \in L^{\infty}_c(dx)$, we consider the operator


\begin{equation}\label{defT}
T_{\alpha}f(x)=\int K_\alpha(x-y)f(y)dy.
\end{equation}


Observe that we do not assume that the operator is bounded.

\begin{obs}Let $1<r<p<n/\alpha$ and $\frac1{q}=\frac1{p}-\frac{\alpha}{n}$. If $f\in L^{\infty}_c(dx)$ and $w^r\in A_{\frac{p}{r},\frac{q}{r}}$ then $T_\alpha f\in L^q(w^q)$.
This remark is a particular case of the Lemma 5.2 in \cite{IFRi18}.
\end{obs}
\begin{obs}This type of operators also appears in several works, for example \cite{BLR11,DPR18,G17}.
\end{obs}


\begin{obs} It can be considered  
that $T_{\alpha}$ is not of convolution type. In this case, we need the H\"ormander and size condition in both variables.
In this paper, we only consider the convolution type operator  and
the general case follows in an analogous way.
\end{obs}

 An intersting example of kernel 
is the following, let consider $L=-\Delta + V$  the Schr\"odinger operator on $\R$, $n\geq 3$, with $V$ satisfies a Reverse-H\"older condition, $RH_q$, with $\frac{n}{2}<q<n$ and let $K$ be the kernel associated to the Riesz transform $L^{-1/2}\nabla$. It can be prove that $K\in S_{0,r'}\cap H_{r'}$, for some $1<r'<\infty$, see details in \cite{GLP08,ACADH17,Li16}. We define $K_{\alpha}(x,y)=|x-y|^{\alpha}K(x,y)$, then by \cite{BLR11}, $K_\alpha\in S_{\alpha,r'}\cap H_{\alpha,r'}$.\\

For $0\leq\alpha<n$, $1\leq r <\infty$ and $f\in
L^1_{\text{loc}}(dx)$, the $M_{\alpha,r}$ maximal operator is
defined by
$$M_{\alpha,r}f(x)=\sup_{B\ni x} |B|^{\alpha/n}\|f\|_{r,B},$$
where the supremum is taken over all the balls $B$ that contains
$x$.

 A more general case of this type  of operator has been studied
by Kurtz in \cite{K90}. We state his results only for the case  we
are considering,
\begin{teo}\cite{K90}\label{KMsharp}
Let $0<\alpha < n$ and $1\leq r<n/\alpha$. Let $K_\alpha\in
H_{\alpha,r'}$  and  suppose $T_\alpha$ is bounded from $L^{r}(dx)$
into $L^q(dx)$ for $\frac1{q}=\frac1{r}-\frac{\alpha}{n}$. Then
there exists a constant $C>0$, such that for  $f\in
L^r_{\text{loc}}(dx)$,
$$M^{\sharp}(T_{\alpha}f)(x)\leq C M_{\alpha,r}f(x),$$
where $M^{\sharp}$ is the classical sharp maximal function.
\end{teo}

\begin{obs} In \cite{K90}, Kurtz defined a class of kernels $K(r,\alpha)$, it is easy to see that if $K_{\alpha}\in K(r',\alpha)$ then
$K_{\alpha}\in H_{\alpha,r'}$ and the operator $T_{\alpha}
f=K_{\alpha}*f$ is bounded from $L^{r}(dx)$ into $L^q(dx)$ for
$\frac1{q}=\frac1{r}-\frac{\alpha}{n}$.
\end{obs}

\begin{teo}\cite{K90}
Let $0<\alpha < n$ and $1\leq r<n/\alpha$. Let $K_\alpha\in
H_{\alpha,s}$ and suppose  $T_\alpha$ is bounded from $L^{s}(dx)$
into $L^q(dx)$  
 for all $(s,q)$ with $\frac1{q}=\frac1{s}-\frac{\alpha}{n}$ and $n/(n-\alpha)<s<r'$.
If $r<p<n/\alpha$,
$\frac1{q}=\frac1{p}-\frac{\alpha}{n}$ and $w^r\in
A_{\frac{p}{r},\frac{q}{r}}$, then there exists a constant $C_w$,
independent of $f$ but depending on $w$, such that
\begin{equation}\label{fuerteKurtz}
\|T_{\alpha}f\|_{L^q(w^q)}\leq C_w \|f\|_{L^p(w^p)}.
\end{equation}
\end{teo}

 More recently, in \cite{BLR11} the authors proved a version of Theorem \ref{KMsharp} without assuming that the operator is bounded. 

\begin{teo}\cite{BLR11}
Let $0<\alpha<n$ and $1<r<\infty$. Let $T$ be defined as  in $\eqref{defT}$ and let
$K_\alpha\in S_{\alpha,r'}\cap H_{\alpha,r'}$. Then   there exists $C>0$  such that for  $f\in L_c^{\infty}$,
$$M^{\sharp}(T_{\alpha}f)(x)\leq C M_{\alpha,r}f(x),$$
where $M^{\sharp}$ is the classical sharp maximal function.
\end{teo}

From this result and the good-$\lambda$ technique, it follows that
\begin{prop}\label{debil}
Let $0<\alpha < n$ and $1\leq r<n/\alpha$. Let    $T$ be defined as  in $\eqref{defT}$ and let
$K_\alpha\in S_{\alpha,r'}\cap H_{\alpha,r'}$, then there
exists a constant $C_w >0$, depending on $w$, such that for  $f\in
L^{\infty}_{c}(dx)$ and $w^r\in A_{1,\frac{n}{n-\alpha r}}$
 $$\underset{\lambda>0}{\sup }\lambda^r {w^{\frac{rn}{n-\alpha r}}\{x\in \R : |(T_{\alpha}f)(x)|>\lambda\}}^{\frac{n-\alpha r}{n}} \leq C_w\int |f|^rw^r. $$
\end{prop}
The idea of  this proof is the same as the one given in
Theorem 3.6 in \cite{RU14}, so we omit it.

From this result we know that if $w^r\in
A_{\frac{p}{r},\frac{q}{r}}$ then $T_{\alpha}$ is bounded from
$L^{p}(w^p)$ into $L^{q}(w^q)$ and it was only  known  the
dependence of the $w$ constant in the case $T_{\alpha}=I_{\alpha}$,
Theorem \ref{SharpIalpha}.
 The main result in this paper is the dependence of the constant
$[w^r]_{A_{\frac{p}{r},\frac{q}{r}}}$ in  the inequality
\eqref{fuerteKurtz},
 for a class  of operators  given by a kernel $K_\alpha$ less regular than the one on $I_\alpha$. These kernels  satisfy a
$L^{\alpha,r'}-$H\"ormander condition. The result is the following

\begin{teo}\label{CotaT}
Let $0<\alpha<n$ and let $T_\alpha $ defined as in \eqref{defT}. Let
$1\leq r<p<n/\alpha$, $1/q=1/p-\alpha/n$. Suppose  $K_\alpha \in
S_{\alpha,r'} \cap H_{\alpha,r'}$. If $w^r\in
A_{\frac{p}{r},\frac{q}{r}}$, then
$$\|T_{\alpha}f\|_{L^q(w^q)}\leq c_n [w^r]_{A_{\frac{p}{r},\frac{q}{r}}}^{\max\left\{1-\frac{\alpha}{n},\frac{(p/r)'}{q}\left(1-\frac{\alpha r}{n}\right)\right\}} \|f\|_{L^p(w^p)}.$$
\end{teo}
This estimate is sharp in the following sense

\begin{prop}\label{CotaSharp}
Let  $0<\alpha<n$, $1\leq r<p<n/\alpha$ and $1/q=1/p-\alpha/n$. Let
$K_\alpha \in S_{\alpha,r'} \cap H_{\alpha,r'}$  and $T_\alpha$
defined as in \eqref{defT}.  Suppose that there exists an  increasing 
function $\Phi:[1,\infty)\rightarrow (0,\infty)$ such that
$$\|T_{\alpha}\|_{L^p(w^p)\rightarrow L^q(w^q)}\lesssim \Phi([w^r]_{A_{\frac{p}{r},\frac{q}{r}}}),$$
for all $w^r\in A_{\frac{p}{r},\frac{q}{r}}$ then
$$\Phi(t)\gtrsim
t^{\max\left\{1-\frac{\alpha}{n},\frac{(p/r)'}{q}\left(1-\frac{\alpha r}{n}\right)\right\}}.$$
\end{prop}

\begin{obs}
In the case of  the fractional integral $I_\alpha$, $r'=\infty$, we
obtain the same sharp bound as in \cite{LMPT10}.
\end{obs}

\begin{obs}
  For the singular integral with kernel $k\in H_{r'}$, ($\alpha=0$)
  K. Li in \cite{Li16} give the sparse domination.  Following the same proof of Proposition
  \ref{CotaSharp} in Section 5,  one can obtain  the sharpness 
in this case.
\end{obs}

The paper continues as follows:  in the next  section we present
some particular operator as applications of these results. In
Section 3 we give the sparse domination for $T_{\alpha}$. In Section
4 we obtain the $L^p(w^p)-L^q(w^q)$ boundedness of the sparse
operator with the 
 dependence of the
$[w^r]_{A_{\frac{p}{r},\frac{q}{r}}}$ constant. Finally in Section 5
we give some examples to prove that the dependency of the constant
given in Section 4 is optimal.

\section{Applications}
In this section, we give more examples of our results.\\

\begin{itemize}
\item Fractional rough operator:\\

 Let  $\Omega$ be a function
defined on $S^{n-1}$. 
We consider its extention to $\R \setminus \{0\}$, 
which is defined as $\Omega(x)=\Omega(x/|x|)$. Thus $\Omega$ is a
homogeneous function of degree $0$. For $1\leq s\leq \infty$, the
$L^{s}$-modulus of continuity of $\Omega$ is defined as
$$\bar{\omega}_{s}(t)=\sup_{|y|<t} \|\Omega( \cdot +y)-\Omega(\cdot)\|_{s,S^{n-1}}.$$

Let  $0<\alpha<n$, $r'>\frac{n}{n-\alpha}$ and  $\Omega\in
L^{r'}(S^{n-1})$ such that $\int_0^1
\bar{\omega}_{r'}(t)\frac{dt}{t}<\infty$. Let
$$K_\alpha(x)=\frac{\Omega(x/|x|)}{|x|^{n-\alpha}},$$
and $T_\alpha f(x)=K_\alpha *f(x)$. In \cite{BLR11}, it is proved
that $K_\alpha\in H_{\alpha,r'}\cap S_{\alpha,r'}$. Since
$r'>\frac{n}{n-\alpha}$, its conjugate exponent 
$r<n/\alpha$.
Thus applying the main result, Theorem \ref{CotaT},  we obtain that
for $1<r<p<n/\alpha$ and $1/q=1/p-\alpha/n$
$$\|T_\alpha f\|_{L^q(w^q)}\leq c_n [w^r]_{A_{\frac{p}{r},\frac{q}{r}}}^{\max\left\{1-\frac{\alpha}{n},\frac{(p/r)'}{q}\left(1-\frac{\alpha r}{n}\right)\right\}} \|f\|_{L^p(w^p)}.$$




\

\item Others kernels: \\

 Let $0<\alpha<1$,
$\beta >0$, $1<r<p<1/\alpha$ and $\frac1{q}=\frac{1}p-\alpha$. For
$r'$ the conjugated exponent of $r$, let us consider
$$k(t)=\left(\frac1{t\log(e/t)^{1+\beta}}\right)^{1/r'}\chi_{(0,1)}(t).$$
As it was shown in  \cite{IFRi18} and \cite{MPTG05}, $k\in H_{r'}\cap
S_{0,r'}$. Now, let
$$K_{\alpha}(t)=|t+4|^{\alpha}k(|t+4|),$$
by Proposition 4.1 in \cite{BLR11}, $K_{\alpha}\in H_{\alpha,r'}\cap
S_{\alpha,r'}$. Finally let  $T_\alpha f=K_\alpha*f$. Applying the
main result, Theorem \ref{CotaT},  we obtain for $1<r<p<1/\alpha$
and  $1/q=1/p-\alpha$
$$\|T_{\alpha}f\|_{L^q(w^q)}\leq c_n [w^r]_{A_{\frac{p}{r},\frac{q}{r}}}^{\max\left\{1-{\alpha},\frac{(p/r)'}{q}\left(1-{\alpha r}\right)\right\}} \|f\|_{L^p(w^p)}.$$
 For more details of the sharpness, see section 5.2.

\end{itemize}

\section{Sparse domination for $T_{\alpha}$}
In this section we present a sparse domination result for the
operator $T_{\alpha}$.  Let us recall some well known results.

It is defined that a kernel $K\in H_{Dini}$ if  
$$|K(x)|\leq \frac{C_K}{|x|^n}$$
and 
$$|K(x-y)-K(x'-y)|\leq \omega\left(\frac{|x-x'|}{|x-y|}\right)\frac1{|x-y|^n}$$
for $|x-y|>2|x-x'|$. The function $\omega:[0,1]\to [0, \infty)$ is a continuous, increasing, submultiplicative with $\omega(0)=0$ and satisfies a the Dini condition, $$\int_{0}^1 \omega(t)\frac{dt}{t}<\infty.$$

Observe that 
$$H_{Dini}\subset H_{\infty}\subset H_r \subset H_s \subset H_1, \quad 1<s<r<\infty.$$

In the case that $T$ is a Calder\'on-Zygmund operator with  $K\in H_{Dini}$ the sparse domination was proved in \cite{L16}. For their commutators in \cite{LORR16} and for vector-valued case in \cite{CLPRR20}. The sparse domination, for $K\in H_r$ was considered in \cite{Li16} and for $K$ satifying a  Young type H\"ormander condition was considered in \cite{IFRR17}. Finally, for the case of $I_{\alpha}$ the sparse domination was studied in \cite{AMPRR17}. 
 It is posible to obtain a pointwise sparse do\-mi\-na\-tion that covers the general fractional operators that we are considering.
\\

To state our result of sparse domination, we recall some definitions.

Given a cube $Q\in \R$, we denote by $\mathcal{D}(Q)$ the family of
all dyadic cubes  respect to $Q$, that is, the cube obtained subdividing
 repeatedly $Q$ and each of its descendant into $2^n$ subcubes of the same side lengths.

Given a dyadic family $\mathcal{D}$ we say that a family $\mathscr{S}\subset \mathcal{D}$ is a $\eta$-sparse family with $0<\eta<1$, if
 for every $Q\in \mathscr{S}$, there exists a measurable set $E_Q\subset Q$ such that $\eta|Q|\leq |E_Q|$ and the family $\{E_Q\}_{Q\in \mathscr{S}}$
  are pairwise disjoint.

\begin{teo}\label{Sparse}
Let $0<\alpha<n$, $1\leq r < \infty$ and let $T_\alpha$ be defined
as in \eqref{defT}. Suppose  $K_\alpha \in S_{\alpha,r'} \cap
H_{\alpha,r'}$. For any $f\in L^{\infty}_c(\R)$,  there exist $3^n$
sparse families   such that for a.e. $x\in \R$,
$$|T_{\alpha}f(x)|\leq c  \sum_{j=1}^{3^n}\sum_{Q\in \mathscr{S}_j}|Q|^{\alpha/n}\|f\|_{r,Q}\chi_{Q}(x):=c  \sum_{j=1}^{3^n}\A_{r,\mathscr{S}_j}^{\alpha}f(x).$$
\end{teo}

 The grand maximal truncated operator $M_{T_{\alpha}}$ is  defined by
$$M_{T_{\alpha}}f(x)=\underset{Q\ni x}{\sup }\underset{\xi \in Q}{\text{ ess}\sup} |T_{\alpha}(f\chi_{\R\setminus 3Q})(\xi)|,$$
where the supremmum is taken over all the cubes $Q\subset \R$
containing $x$. For the proof of the preceding theorem we need to
show that $M_{T_{\alpha}}$ maps $L^r(dx)$ into
$L^{\frac{rn}{n-\alpha r},\infty}(dx)$.
 Also we need:

\begin{itemize}

\item a local version of  $M_{T_{\alpha}}$ 
which is defined, for a cube $Q_0\subset \R$,
 in the following way:

$$M_{T_{\alpha},Q_0}f(x)=\underset{x\in Q\subset Q_0}{\sup}\underset{\xi \in Q}{\text{ ess}\sup} |T_{\alpha}(f\chi_{3Q_0 \setminus 3Q})(\xi)|.$$

\item  Let $K_\alpha \in  S_{\alpha,r'}\cap H_{\alpha, r'}$  we
define $$\tilde T_\alpha f(x)=\int |K_\alpha|(x-y)f(y)dy.$$ Observe
that if $K_\alpha \in  S_{\alpha,r'}\cap H_{\alpha, r'}$ then
$|K_\alpha |\in  S_{\alpha,r'}\cap H_{\alpha, r'}$ and Proposition
\ref{debil} holds for $\tilde T_\alpha$.

\end{itemize}

\begin{lema}\label{LemaSparse}
Let $0<\alpha<n$, $1\leq r <\infty$, $K_\alpha \in S_{\alpha,r'}
\cap H_{\alpha,r'}$ and $Q_0\subset \R$ be a cube. Let $T_\alpha$ be
defined as in \eqref{defT}  and $f\in L^{\infty}_c(\mathbb{R}^{n})$. Then,
\begin{enumerate}
\item for $a.e.\, x\in Q_0$
$$|T_{\alpha}(f\chi_{3Q_0})(x)|\leq M_{T_{\alpha},Q_0}f(x),$$
\item for all $x\in \R$
$$M_{T_{\alpha}}(f)(x)\lesssim M_{\alpha,r}f(x) + \tilde T_{\alpha}(|f|)(x).$$
\end{enumerate}
From the last estimate and  Proposition \ref{debil} it follows that
$M_{T_{\alpha}}$ is bounded from $L^r(dx)$ into
$L^{\frac{rn}{n-\alpha r},\infty}(dx)$.
\end{lema}
\begin{proof}
\begin{enumerate}
\item Let $Q(x,s)$ a cube centered at $x$ with side length $s$ such that $Q(x,s)\subset Q_0$, then
\begin{align*}
|T_{\alpha}(f\chi_{3Q_0})(x)|&\leq
|T_{\alpha}(f\chi_{3Q(x,s)})(x)|+|T_{\alpha}(f\chi_{3Q_0\setminus3Q(x,s)})(x)|.
\end{align*}
For the first term, let us consider  $B(x,R)$ with $R=3\sqrt{n}s$
then $3Q(x,s)\subset B(x,R)$. As $K_\alpha \in S_{\alpha,r'}$ we
have

\begin{align*}
|T_{\alpha}(f\chi_{3Q(x,s)})(x)|& 
\leq \int_{B(x,R)}|K_{\alpha}(x-y)||f(y)|dy
\\&=\sum_{m=0}^{\infty}\frac{|B(x,2^{-m}R)|}{|B(x,2^{-m}R)|}\int_{B(x,2^{-m}R)} \chi_{B(x,2^{-m}R)\setminus B(x,2^{-m-1}R)}|K_{\alpha}(x-y)||f(y)|dy
\\&\leq \sum_{m=0}^{\infty}|B(x,2^{-m}R)| \|K_{\alpha}\|_{r',|x|\sim2^{-m-1}R}\|f\|_{r,B(x,2^{-m}R)}
\\&\leq c M_{r}(f)(x)\sum_{m=0}^{\infty}(2^{-m}R)^n(2^{-m}R)^{\alpha-n}
\\&=c M_{r}(f)(x)R^{\alpha}\sum_{m=0}^{\infty}(2^{-m})^{\alpha}
=c M_{r}(f)(x)R^{\alpha}.
\end{align*}

Then,
$$|T_{\alpha}(f\chi_{3Q_0})(x)|\leq c_n s^{\alpha}M_{r}f(x) + M_{T_{\alpha},Q_0}f(x).$$
Observe that by hypothesis $M_rf<\infty$, then letting $s\rightarrow
0$, we obtain the desired estimate.

\item Let $x\in \mathbb{R}^n$ and let $Q$ be a cube containing $x$. Let $B_x$ be a ball with radius $R$ such that $3Q\subset B_x$.
For every $\xi \in Q$, we have
\begin{align*}
|T_{\alpha}(f\chi_{\R\setminus 3Q})(\xi)|&\leq |T_{\alpha}(f\chi_{\R\setminus B_x})(\xi)-T_\alpha(f\chi_{\R\setminus B_x})(x)|
+|T_{\alpha}(f\chi_{B_x \setminus 3Q})(\xi)| + |T_{\alpha}(f\chi_{\R\setminus B_x})(x)|
\\&\lesssim  |T_{\alpha}(f\chi_{\R\setminus B_x})(\xi)-T_{\alpha}(f\chi_{\R\setminus B_x})(x)| + |T_{\alpha}(f\chi_{B_x \setminus 3Q})(\xi)|+
 \tilde T_{\alpha}(|f|)(x).
\end{align*}
For the first term, as $K_\alpha\in H_{\alpha,r'}$, we get
\begin{align*}
|T_{\alpha}(f\chi_{\R\setminus B_x})(\xi)-T_{\alpha}(f\chi_{\R\setminus B_x})(x)|&\leq \int_{\R\setminus B_x} |K_{\alpha}(\xi-y)-K_{\alpha}(x-y)||f(y)|dy
\\&= \sum_{m=1}^{\infty} \frac{|2^mB_x|}{|2^mB_x|}\int_{2^{m+1}B_x\setminus2^mB_x}|K_{\alpha}(\xi-y)-K_{\alpha}(x-y)||f(y)|dy
\\&\leq \sum_{m=1}^{\infty} (2^mR)^n \|K_{\alpha}(\xi-\cdot)-K_{\alpha}(x-\cdot)\|_{r',|y|\sim 2^mR}\|f\|_{r,2^{m+1}B_x}
\\&\leq \sum_{m=1}^{\infty} (2^mR)^{n-\alpha} \|K_{\alpha}(\xi-\cdot)-K_{\alpha}(x-\cdot)\|_{r',|y|\sim 2^mR} M_{\alpha,r}f(x)
\\&\leq c_r M_{\alpha,r}f(x).
\end{align*}
For the second term, observe that there exists $l\in \N$ such that
$B(x,2^{-l}R)\subset 3Q$, then, as  $K_\alpha\in
S_{\alpha,r'}$, we obtain
\begin{align*}
|T_{\alpha}(f\chi_{B_x \setminus 3Q})(\xi)|&\leq \int_{B_x \setminus 3Q} |K_{\alpha}(x-y)||f(y)|dy
\\&\leq \sum_{m=0}^{l-1}\int_{B(x,2^{-m}R) \setminus B(x,2^{-m-1}R)} |K_{\alpha}(x-y)||f(y)|dy
\\&\leq  \sum_{m=0}^{l-1}|B(x,2^{-m}R)| \|K_{\alpha}\|_{r',|x|\sim2^{-m-1}R}\|f\|_{r,B(x,2^{-m}R)}
\\&\leq  c \sum_{m=0}^{l-1}(2^{-m}R)^n(2^{-m}R)^{\alpha-n}\|f\|_{r,B(x,2^{-m}R)}
\\&\leq c M_{\alpha,r}f(x).
\end{align*}
Finally we get
\begin{align*}
|T_{\alpha}(f\chi_{\R\setminus 3Q})(\xi)| &\lesssim M_{\alpha,r}f(x)
+ \tilde T_{\alpha}(|f|)(x).
\end{align*}
\end{enumerate}
\end{proof}

The following lemma is the so called $3^n$ dyadic lattices trick.
This result was es\-ta\-bli\-shed in \cite{LN18} and affirms:
\begin{lema}\label{lemaDiadic}\cite{LN18}
Given a dyadic family $\mathcal{D}$ there exist $3^n$ dyadic families $\mathcal{D}_j$ such that
$$\{3Q:Q\in \mathcal{D}\}=\bigcup_{j=1}^{3^n}\mathcal{D}_j,$$
and for every cube $Q\in \mathcal{D}$ we can find a cube $R_Q$ in each $\mathcal{D}_j$ such that $Q\subset R_Q$ and $3l_Q=l_{R_Q}$.
\end{lema}

\begin{proof}[Proof of Theorem \ref{Sparse}]
We claim that for any cube $Q_0 \in \R$, there exists a
$\frac1{2}\,$-sparse family $\mathcal{F}\subset \mathcal{D}(Q_0)$
such that for $a.e. x\in Q_0$,
\begin{equation}\label{claim}
|T_{\alpha}(f\chi_{3Q_0})(x)|\lesssim  \sum_{Q\in
\mathscr{F}}|3Q|^{\alpha/n}\|f\|_{r,3Q}\chi_{Q}(x).
\end{equation}
Suppose that we have already proved the claim (\ref{claim}). Let us take a partition of $\R$ by cubes $Q_j$ such
 that $\text{supp}(f)\subset 3Q_j$ for each $j$. We can do it as follows. We start with a cube $Q_0$ such
 that $\text{supp}(f)\subset Q_0$, and cover $3Q_0\setminus Q_0$ by $3^n - 1$ congruent cubes $Q_j$, each of
 them satisfies $Q_0\subset 3Q_j$. We do the same for
$9Q_0\setminus 3Q_0$ and so on. The union of all  those cubes will
satisfy the desired properties.

We apply the claim (\ref{claim}) to each cube $Q_j$. Then we have that since $\text{supp}(f)\subset 3Q_j$ the following estimate holds $a.e.\,x\in Q_j$
$$|T_{\alpha}f(x)|\chi_{Q_j}(x)=|T_{\alpha}(f\chi_{3Q_0})(x)|\lesssim  \sum_{Q\in \mathcal{F}_j}|3Q|^{\alpha/n}\|f\|_{r,3Q}\chi_{Q}(x),$$
 where each $\mathcal{F}_j\subset\mathcal{D}(Q_j)$ is a $\frac1{2}\,$-sparse
family. Taking $\mathcal{F}=\bigcup_j \mathcal{F}_j$, we have that
$\mathcal{F}$  is a $\frac1{2}\,$-sparse family and for $a.e. x\in
\R$
$$|T_{\alpha}f(x)|\lesssim  \sum_{Q\in \mathcal{F}}|3Q|^{\alpha/n}\|f\|_{r,3Q}\chi_{Q}(x).$$

From Lemma \ref{lemaDiadic} it follows that there exist $3^n$ dyadic
families such that for every cube $Q$ of $\R$ there is a cube
$R_Q\in \mathcal{D}_j$ for some $j$ for which $3Q\subset R_Q$ and
$|R_Q|\leq 3^n|3Q|$. Setting
$$\mathscr{S}_j=\{R_Q\in D_j : Q\in \mathcal{F}\},$$
and since $\mathcal{F}$ is a $\frac1{2}\,$-sparse, we obtain that
for each family $\mathscr{S}_j$ is $\frac1{2.9^n}$-sparse. Then we
have that
$$|T_{\alpha}f(x)|\lesssim  \sum_{j=1}^{3^n}\sum_{Q\in \mathscr{S}_j}|Q|^{\alpha/n}\|f\|_{r,Q}\chi_{Q}(x).$$

Proof of  claim (\ref{claim}).  To prove the claim  it is suffice to
show the following recursive estimate: there exists a countable
family $\{P_j\}_j$ of pairwise disjoint cube in $\mathcal{D}(Q_0)$
such that $\sum_j P_j\leq \frac1{2}|Q_0|$ and
\begin{equation}\label{claim2}
|T_{\alpha}(f\chi_{3Q_0})(x)|\chi_{Q_0}(x)\leq c
|3Q_0|^{\alpha/n}\|f\|_{r,3Q_0}\chi_{Q_0}(x) + \sum_j
|T_{\alpha}(f\chi_{3P_j})(x)|\chi_{P_j}(x),
\end{equation}
for $a.e.\,x\in Q_0$. Iterating this estimate we obtain
(\ref{claim})  with $\mathcal{F}$ being the union of all the
families $\{P_j^k\}$ where $\{P_j^0\}=\{Q_0\}$, $\{P_j^1\}=\{P_j\}$
and the $\{P_j^k\}$ are the cubes obtained at the $k$-th stage of
the iterative process. It is also clear that $\mathcal{F}$ is a
$\frac1{2}\,$-sparse family. Indeed, for each $P_j^k$ it is suffices
to choose
$$E_{P_j^k}=P_j^k\setminus \bigcup_j P_j^{k+1}.$$
Let us prove  the recursive estimate (\ref{claim2}). Observe that for any family  $\{P_j\}\subset \mathcal{D}(Q_0)$ of disjoint cubes, we have
\begin{align*}
|T_{\alpha}&(f\chi_{3Q_0})(x)|\chi_{Q_0}(x)\leq |T_{\alpha}(f\chi_{3Q_0})(x)|\chi_{Q_0\setminus\cup_j P_j}(x) + \sum_j |T_{\alpha}(f\chi_{3Q_0})(x)|\chi_{ P_j}(x)
\\&\leq |T_{\alpha}(f\chi_{3Q_0})(x)|\chi_{Q_0\setminus \cup_j P_j}(x) + \sum_j |T_{\alpha}(f\chi_{3Q_0\setminus 3P_j})(x)|\chi_{ P_j}(x) + \sum_j |T_{\alpha}(f\chi_{3P_j})(x)|\chi_{ P_j}(x),
\end{align*}
for almost every $x\in \R$. So it is suffice to show that we can
choose a countable family $\{P_j\}_j$ of pairwise disjoint cube in
$\mathcal{D}(Q_0)$ such that $\sum_j P_j\leq \frac1{2}|Q_0|$ and
for, $a.e. x\in Q_0$ we have,
\begin{equation}\label{claim3}
|T_{\alpha}(f\chi_{3Q_0})(x)|\chi_{Q_0\setminus \cup_j P_j}(x) + \sum_j |T_{\alpha}(f\chi_{3Q_0\setminus 3P_j})(x)|\chi_{ P_j}(x)\lesssim  |3Q_0|^{\alpha/n}\|f\|_{r,3Q_0}\chi_{Q_0}(x).
\end{equation}
Now we define the set $E$ as
\begin{align*}
E=
\{x\in Q_0 : M_{T_{\alpha},Q_0}f(x)> \beta_n c |3Q_0|^{\alpha/n}\|f\|_{r,3Q_0}\},
\end{align*}
 by Lemma \ref{LemaSparse} we can choose $\beta_n$ such that $|E|\leq
\frac1{2^{n+2}}|Q_0|$.


We apply Calder\'on-Zygmund decomposition to the function $\chi_E$
on $Q_0$ at height $\lambda =\frac1{2^{n+1}}$. Then, there exist a
family $\{P_j\}\subset \mathcal{D}(Q_0)$ of pairwise disjoint cubes
such that
$$\left\{x\in Q_0: \chi_E(x)> \frac1{2^{n+1}}\right\}
=\bigcup_j P_j.$$

From this it follows that $|E\setminus \cup_j P_j|=0$,
$$\sum_j |P_j|\leq 2^{n+1}|E|\leq \frac1{2}|Q_0|,$$
and
$$\frac1{2^{n+1}}\leq \frac{|P_j\cap E|}{|P_j|}\leq \frac1{2},$$
from which it follows that $|P_j\cap E^c|>0$.

Since $P_j\cap E^c\not= \emptyset$, we have
$M_{T_{\alpha},Q_0}(f)(x)\leq \beta_n c
|3Q_0|^{\alpha/n}\|f\|_{r,3Q_0}$ for some $x\in P_j$ and this
implies that
$$\underset{\xi\in P_j}{\text{ess}\sup}|T_\alpha(f\chi_{3Q_0\setminus 3P_j})(\xi)|\leq \beta_n c |3Q_0|^{\alpha/n}\|f\|_{r,3Q_0},$$
which allows us to control the second  term in (\ref{claim3}).

By (1) in Lemma \ref{LemaSparse}, for  $a.e. x\in Q_0$ we have
$$|T_{\alpha}(f\chi_{3Q_0})(x)|\chi_{Q_0\setminus \cup_j P_j}(x)\leq M_{T_{\alpha},Q_0}f(x)\chi_{Q_0\setminus \cup_j P_j}(x).$$
Since $|E\setminus \cup_j P_j|=0$ and by the definition of $E$, for
$a.e. x\in Q_0\setminus \cup_j P_j$ we obtain
$$M_{T_{\alpha},Q_0}(f)(x)\leq \beta_n c |3Q_0|^{\alpha/n}\|f\|_{r,3Q_0}.$$
Then, for $a.e. x\in Q_0\setminus \cup_j P_j$ we get
$$|T_{\alpha}(f\chi_{3Q_0})(x)|\leq \beta_n c |3Q_0|^{\alpha/n}\|f\|_{r,3Q_0}.$$
Thus we obtain the estimate in (\ref{claim3}).

\end{proof}

\section{Sharp bounds for norm inequality}

Since the sparse domination is a pointwise estimate, is it suffice
to prove Theorem \ref{CotaT} and Proposition \ref{CotaSharp} for the
sparse operator $A_{r,\mathscr{S}}^{\alpha}$ for any sparse family
$\mathscr{S}$.

\begin{teo}\label{cotasharp}Let $0\leq \alpha<n$, $1\leq r<p<n/\alpha$ and $1/q=1/p-\alpha/n$. If $w^r\in A_{\frac{p}{r},\frac{q}{r}}$, then
$$\|A_{r,\mathscr{S}}^{\alpha}f\|_{L^q(w^q)}\leq c_n [w^r]_{A_{p/r,q/r}}^{\max\left\{1-\frac{\alpha}{n},\frac{(p/r)'}{q}\left(1-\frac{\alpha r}{n}\right)\right\}} \|f\|_{L^p(w^p)},$$
this estimate is sharp in the following sense:

 If there exists a increasing function
$\Phi:[1,\infty)\rightarrow (0,\infty)$   such
that
$$\|A_{r,\mathscr{S}}^{\alpha}\|_{L^p(w^p)\rightarrow L^q(w^q)}\lesssim \Phi([w^r]_{A_{\frac{p}{r},\frac{q}{r}}}),$$
for all $w^r\in A_{\frac{p}{r},\frac{q}{r}}$, then
$$\Phi(t)\gtrsim
t^{\max\left\{1-\frac{\alpha}{n},\frac{(p/r)'}{q}\left(1-\frac{\alpha
r}{n}\right)\right\}}.$$
\end{teo}

\begin{obs}
The first approximation of this type for  the fractional integral,
using the sparse technique, appears in \cite{CU15}. In this paper
the author does not prove the sharpness of the constant. In the case
$r=1$, the appropriate  sparse operator  for the fractional integral
$I_\alpha$, we obtain the same sharp bound as in \cite{AMPRR17}. If
$\alpha=0$, we get the same sharp bound as in \cite{L16}.
\end{obs}

We consider the following sparse operator defined in \cite{FH18}, for $\mathscr{S}$ a sparse family, $0<s<\infty$ and $0<\beta\leq 1$

$$\tilde{A}_{s,\mathscr{S}}^{\beta}g(x)=\left(\sum_{Q\in\mathscr{S}} \left(|Q|^{-\beta}\int_Q g\right)^s\chi_{Q}(x)\right)^{1/s}.$$

\begin{teo}\cite{FH18} Let $1\leq r <p \leq q <\infty$, and $0<\beta \leq 1$. Let consider the weights $u,\sigma\in A_{\infty}$ .
The sparse operator $\tilde{A}_{s,\mathscr{S}}^{\beta}(\cdot \sigma)$ maps $L^p(\sigma)\rightarrow L^q(u)$ if and only if
the two-weight $A_{p,q}^{\beta}$-characteristic
$$[u,\sigma]_{A_{p,q}^{\beta}(\mathscr{S})}:=\underset{Q\in\mathscr{S}}{\sup}|Q|^{-\beta}u(Q)^{1/q}\sigma(Q)^{1/p'},$$
is finite, and in this case

$$1\leq \frac{\|\tilde{A}_{s,\mathscr{S}}^{\beta}(\cdot \sigma)\|_{L^p(\sigma)\rightarrow L^q(u)}}{[u,\sigma]_{A_{p,q}^{\beta}(\mathscr{S})}}\lesssim [\sigma]_{A_{\infty}}^{1/q}+[u]_{A_{\infty}}^{\frac1{s}-\frac1{p}}.$$
\end{teo}

\begin{proof}[Proof of Theorem \ref{cotasharp}]
Let $\sigma=w^{-(p/r)'r}$. Observe that
$$A_{r,\mathscr{S}}^{\alpha}(f)=\left(\tilde{A}_{1/r,\mathscr{S}}^{1-\alpha/n}(f^r)\right)^{1/r}.$$
Then,
\begin{align*}
\|A_{r,\mathscr{S}}^{\alpha}(f)\|_{L^q(w^q)}^r&=\|\tilde{A}_{1/r,\mathscr{S}}^{1-\alpha/n}(f^r)\|_{L^{q/r}(w^q)}=\|\tilde{A}_{1/r,\mathscr{S}}^{1-\alpha/n}(f^r\sigma^{-1}\sigma)\|_{L^{q/r}(w^q)}
\\&\lesssim [w^q,\sigma]_{A_{\frac{p}{r},\frac{q}{r}}^{1-\alpha/n}(\mathscr{S})}\left([\sigma]_{A_{\infty}}^{r/q}+[w^q]_{A_{\infty}}^{r-\frac{r}{p}}\right) \|f^r\sigma^{-1}\|_{L^{p/r}(\sigma)}.
\end{align*}
Now observe that,
$$[w^q,\sigma]_{A_{\frac{p}{r},\frac{q}{r}}^{1-\alpha/n}(\mathscr{S})}\leq [w^r]_{A_{p/r,q/r}}^{r/q},$$

and
$$\|f^r\sigma^{-1}\|_{L^{p/r}(\sigma)}=\|f\|_{L^{p}(w^p)}^r.$$
Since
$$[\sigma]_{A_{1+(p/r)'r/q}}=[w^r]_{A_{p/r,q/r}}^{(p/r)'r/q}\qquad [w^q]_{A_{1+\frac{q}{r(p/r)'}}}=[w^r]_{A_{p/r,q/r}},$$

we have
\begin{align*}
\|A_{r,\mathscr{S}}^{\alpha}(f)\|_{L^q(w^q)}&\lesssim  [w^r]_{A_{p/r,q/r}}^{1/q}\left([\sigma]_{A_{\infty}}^{r/q}+[w^q]_{A_{\infty}}^{r-\frac{r}{p}}\right)^{1/r} \|f\|_{L^{p}(w^p)}
\\&\leq  [w^r]_{A_{p/r,q/r}}^{1/q}\left([w^r]_{A_{p/r,q/r}}^{(p/r)'(r/q)^2}+[w^r]_{A_{p/r,q/r}}^{\frac{r}{p'}}\right)^{1/r} \|f\|_{L^{p}(w^p)}
\\&\leq  [w^r]_{A_{p/r,q/r}}^{1/q+\max\{(p/r)'r/q^2,\frac1{p'}\}} \|f\|_{L^{p}(w^p)}
\\&\leq  [w^r]_{A_{p/r,q/r}}^{\max\{(1-\frac{\alpha r}{n})(p/r)'/q,1-\alpha/n\}}
\|f\|_{L^{p}(w^p)},
\end{align*}
where the last inequality holds  since
$(1+(p/r)'r/q)=(1-\frac{\alpha r}{n})(p/r)'$ and
$1/q+1/p'=1-\alpha/n$.

\end{proof}

\section{Examples}
\subsection{Sparse operator $A_{r,\mathscr{S}}^{\alpha}$}
 In this subsection, we prove the sharpness of Theorem
\ref{cotasharp}.

\begin{proof}
Let $A=A_{r,\mathscr{S}}^{\alpha}$ the sparse operator. Let
$0<\varepsilon<1$. If
$$w_{\varepsilon}(x)=|x|^{\frac{n-\varepsilon}{r(p/r)'}} \qquad \text{ and }\quad  f(x)=|x|^{\frac{\varepsilon - n}{r}}\chi_{B(0,1)},$$
then
$$[w_{\varepsilon}^r]_{A_{\frac{p}{r},\frac{q}{r}}}\simeq \varepsilon^{-\frac{q}{r(p/r)'}} \qquad \text{ and }\quad  \|fw_{\varepsilon}\|_{L^p}\simeq \varepsilon^{-1/p}.$$

Let $\{Q_k\}$ the cube of center $0$ and length $2^{-k}$ and observe
that $B(0,1)\subset Q_0$.  This family  $\{Q_k\}$ is a
$\frac1{2}\,$-sparse family  with  $E_{Q_k}=Q_k\setminus Q_{k+1}$. 

Now, if $x\in E_{Q_k}$, $k\in \N$
\begin{align*} Af(x)&\geq
|Q_{k}|^{\alpha/n-1/r}\left(\int_{Q_k}|y|^{{\varepsilon -
n}}\right)^{1/r}
\gtrsim (2^{-kn})^{\alpha/n-1/r}\left(\frac{2^{-k\varepsilon}}{\varepsilon}\right)^{1/r}
\gtrsim \varepsilon^{-1/r} |x|^{\alpha-n/r+\varepsilon/r}.
\end{align*}

Therefore,
\begin{align*}
\int Af^q w_{\varepsilon}^q &\geq \sum_{k=1}^{\infty} \int_{E_{Q_k}}Af^q w_{\varepsilon}^q
\\&\gtrsim  \varepsilon^{-q/r} \int_{B(0,\frac1{2})} |x|^{q(\alpha-n/r+\varepsilon/r)+q\frac{n-\varepsilon}{r(p/r)'}}dx
\\&\simeq   \varepsilon^{-q/r-1},
\end{align*}
since $q(\alpha/n-1/r+\varepsilon/r)+q\frac{n-\varepsilon}{r(p/r)'}\leq
\varepsilon q/p -n$. Then

$$\varepsilon^{-\frac1{r(p/r)'}-1/q}\lesssim \varepsilon^{-1/r-1/q+1/p}\lesssim \|A\|_{L^p(w_{\varepsilon}^p)\rightarrow L^q(w_{\varepsilon}^q)}\lesssim \Phi(\varepsilon^{-\frac{q}{r(p/r)'}}).$$

Now, take $t=\varepsilon^{-\frac{q}{r(p/r)'}}$,
$$\Phi(t)\gtrsim t^{(p/r)'1/q(1-\alpha r/n)}.$$
Let $0<\varepsilon<1$. If
$$w_{\varepsilon}(x)=|x|^{\frac{\varepsilon-n}{q}}\qquad \text{ and }\quad f(x)=|x|^{\frac{\varepsilon - n}{r}}\chi_{B(0,1)},$$
then
$$[w_{\varepsilon}^r]_{A_{\frac{p}{r},\frac{q}{r}}}\simeq \varepsilon^{-1} \qquad \text{ and }\quad \|fw_{\varepsilon}\|_{L^p}\lesssim \varepsilon^{-1/p}.$$

Since $1/r+1/q=1/r-\alpha/n+1/p\geq 1/p,$
\begin{align*}
\int f^pw_{\varepsilon}^p&=\int_{B(0,1)} |x|^{\frac{\varepsilon - n}{r}p+\frac{\varepsilon-n}{q}p}=\int_{B(0,1)} |x|^{p\left(\varepsilon(1/r+1/q)-n(1/r+1/q)\right)}
\\&\leq \int_{B(0,1)} |x|^{p\left(\varepsilon(1/r+1/q)-n/p\right)}\simeq
\varepsilon^{-1}.
\end{align*}

Now, if $x\in Q_0$,
\begin{align*}
Af(x)&\geq |Q_{0}|^{\alpha/n-1/r}\left(\int_{Q_0}|y|^{{\varepsilon -
n}}\right)^{1/r}
\gtrsim \left(\frac1{\varepsilon}\right)^{1/r}\simeq
\varepsilon^{-1/r}\gtrsim \varepsilon^{-1}.
\end{align*}
Then, since $B(0,1)\subset Q_0$
\begin{align*}
\int Af^q w_{\varepsilon}^q &\gtrsim  \varepsilon^{-q} \int_{B(0,1)}
|x|^{\varepsilon-n}dx \gtrsim \varepsilon^{-q-1},
\end{align*}
then,
$$\varepsilon^{-1-1/q}\lesssim \|Af\|_{L^q(w_{\varepsilon}^q)}\lesssim \Phi(\varepsilon^{-1})\|f\|_{L^p(w_{\varepsilon}^p)}\lesssim \Phi(\varepsilon^{-1})\varepsilon^{-1/p}.$$
Now, if we take  $t=\varepsilon^{-1}$ then
$$t^{1-\alpha/n}\lesssim\Phi(t).$$
\end{proof}

\subsection{An operator $T_\alpha$}
 In this subsection, we give an example of an operator to proof the sharpness of the Proposition \ref{CotaSharp}.

\begin{proof}[Proof of Proposition 1.9]  
 Let $0<\alpha<1$,  $\beta>\max\{0, q/r' -1\}$, $1<r<p<1/\alpha$, $\frac1{q}=\frac{1}p-\alpha$ and $\frac1{r}+\frac1{r'}=1$.
Let us consider
$$k(t)=\left(\frac1{t\log(e/t)^{1+\beta}}\right)^{1/r'}\chi_{(0,1)}(t).$$
As shown in  \cite{IFRi18} and \cite{MPTG05}, we know that $k\in
H_{r'}\cap S_{r'}$. Now, let
$$K_{\alpha}(t)=|t+4|^{\alpha}k(|t+4|),$$
by Proposition 4.1 in \cite{BLR11}, $K_{\alpha}\in H_{\alpha,r'}\cap
S_{\alpha,r'}$. Let us consider  $T_\alpha f=K_\alpha*f$.

Observe that there exists $0<t_0<1$ such that $k$ is decreasing in
$(0,t_0)$.


Let $0<\varepsilon<1$. If
$w_{\varepsilon}(x)=|x|^{\frac{1-\varepsilon}{r(p/r)'}}\text{ and }
f(x)=|x+4|^{\frac{\varepsilon}{r} - 1}\chi_{(-5,-3)}(x),$ then
$$[w_{\varepsilon}^r]_{A_{\frac{p}{r},\frac{q}{r}}}\simeq
\varepsilon^{-\frac{q}{r(p/r)'}} \quad \text{ and } \quad
\|fw_{\varepsilon}\|_{L^p}\lesssim \varepsilon^{-1/p}.$$

Observe that if $|x-y|\leq 1$ and $|y|\leq 1$ then $|x|\leq2$ and
supp$(Tf)\subset [-2,2]$.  Let $x \in \;$supp$(Tf)$, 
$|x-y|\leq 1$,
and $0\leq|y|\leq |x|/2$ then $\frac{1}{2}|x|\leq |x-y|
\leq\frac{3}{2}|x|$ and $|x-y|^{\alpha}\gtrsim |x|^{\alpha}$.

For $|x|\leq \frac{2}{3}t_0\leq 2$, since $k$ is decreasing in $(0,t_0)$ we have
\begin{align*}
T_\alpha f(x)&
 \geq \int_{|y|\leq |x|/2}
|x-y|^{\alpha-1/r'}\left(\frac1{\log(e/|x-y|)}\right)^{\frac{1+\beta}{r'}}\chi_{(0,1)}(|x-y|)
|y|^{\frac{\varepsilon}{r} - 1}dy
\\& \gtrsim |x|^{\alpha}k\left(\frac{3}{2}|x|\right) \int_{|y|\leq |x|/2}|y|^{\frac{\varepsilon}{r} - 1}dy
\\&\gtrsim
\varepsilon^{-1}|x|^{\alpha+\frac{\varepsilon}{r}}k\left(\frac{3}{2}|x|\right).
\end{align*}
Then, using $\log(t)<\frac{t^{\varepsilon}}{\varepsilon}$ for $\varepsilon>0$ and $t>1$, we get
\begin{align*}
\int_{\mathbb{R}} |Tf(x)|^qw_{\varepsilon}^q(x)dx&\gtrsim
\varepsilon^{-q}\int_{|x|\leq \frac{2}{3}t_0}
|x|^{q\left(\alpha+\frac{\varepsilon}{r}\right)} k\left(\frac{3}{2}|x|\right)^q |x|^{q\left(\frac{1-\varepsilon}{r(p/r)'}\right)}dx
\\& = \varepsilon^{-q}\int_{|x|\leq  \frac{2}{3}t_0} k\left(\frac{3}{2}|x|\right)^q
|x|^{\frac{q}{r}-1}|x|^{\varepsilon\frac{q}{p}}dx
\\&  \gtrsim \varepsilon^{-q}\int_{|x|\leq \frac{2}{3}t_0} \left(\frac{|x|^{\varepsilon}}{\varepsilon}\right)^{\frac{q}{r'}(1+\beta)}
|x|^{\frac{q}{r}+\varepsilon\frac{q}{p}-1-\frac{q}{r'}}dx
\\&\gtrsim \varepsilon^{-q-1}\int_{|x|\leq  \frac{2}{3}t_0}
|x|^{\frac{q}{r}-\frac{q}{r'}+\varepsilon\left(
\frac{q}{r'}(1+\beta)+\frac{q}{p}\right)-1}dx,
\end{align*}
 the last inequality holds since $\beta>\max\{0, q/r' -1\}$.

If  $r\geq2$,
$$\frac{q}{r}-\frac{q}{r'}=q\left(\frac1{r}-\frac1{r'}\right)\leq 0.$$
Then
\begin{align*}
\int_{\mathbb{R}} |Tf(x)|^qw_{\varepsilon}^q(x)dx&
\gtrsim \varepsilon^{-q-1}\int_{|x|\leq  \frac{2}{3}t_0}
|x|^{\frac{q}{r}-\frac{q}{r'}+\varepsilon\left( \frac{q}{r'}(1+\beta)+\frac{q}{p}\right)-1}dx
\\&\gtrsim \varepsilon^{-q-1}\int_{|x|\leq  \frac{2}{3}t_0}
|x|^{\varepsilon\left( \frac{q}{r'}(1+\beta)+\frac{q}{p}\right)-1}dx
\\&\gtrsim \varepsilon^{-q-2}\geq \varepsilon^{-q-1}.
\end{align*}

If  $r<2$,
$$\frac{q}{r}-\frac{q}{r'}=q\left(\frac1{r}-\frac1{r'}\right)> 0.$$
Then
\begin{align*}
\int_{\mathbb{R}} |Tf(x)|^qw_{\varepsilon}^q(x)dx&
\gtrsim \varepsilon^{-q-1}\int_{|x|\leq  \frac{2}{3}t_0}
|x|^{\frac{q}{r}-\frac{q}{r'}+\varepsilon\left( \frac{q}{r'}(1+\beta)+\frac{q}{p}\right)-1}dx
\\&\gtrsim \varepsilon^{-q-1}\frac{t_0^{\frac{q}{r}-\frac{q}{r'}+\varepsilon\left( \frac{q}{r'}(1+\beta)+\frac{q}{p}\right)}}{\frac{q}{r}-\frac{q}{r'}+\varepsilon\left( \frac{q}{r'}(1+\beta)+\frac{q}{p}\right)}
\\&\gtrsim \varepsilon^{-q-1}.
\end{align*}
 Then, for $1<r<\infty$ we obtain
$$\int_{\mathbb{R}} |Tf(x)|^qw_{\varepsilon}^q(x)dx\gtrsim \varepsilon^{-q-1}.$$
Therefore,
$$\|Tf\|_{L^q(w_{\varepsilon}^{q})}\gtrsim \varepsilon^{-1-1/q}.$$
Then,
$$\varepsilon^{-\frac1{r(p/r)'}-1/q}\lesssim \varepsilon^{-1-1/q+1/p}\lesssim \|T\|_{L^p(w_{\varepsilon}^p)\rightarrow L^q(w_{\varepsilon}^q)}\lesssim \Phi(\varepsilon^{-\frac{q}{r(p/r)'}}).$$
and
$$\Phi(t)\gtrsim t^{(p/r)'1/q(1-\alpha r)}.$$
Let $0<\varepsilon<1$. If $
w_{\varepsilon}(x)=|x|^{\frac{\varepsilon-1}{q}} \text{ and }
f(x)=|x+4|^{\varepsilon /r- 1}\chi_{(-5,-3)},$ then
$$[w_{\varepsilon}^r]_{A_{\frac{p}{r},\frac{q}{r}}}\simeq
\varepsilon^{-1} \quad  \text{ and }\quad
\|fw_{\varepsilon}\|_{L^p}\lesssim \varepsilon^{-1/p}.$$ In an
analoguous way we obtain for $0<|x|<\frac{2}{3}t_0$, $Tf(x)\gtrsim
\varepsilon^{-1}|x|^{\alpha+\frac{\varepsilon}{r}}k\left(\frac{3}{2}|x|\right)$
and
$$\|Tf\|_{L^q(w_{\varepsilon}^{q})}\gtrsim \varepsilon^{-1-1/q}.$$
Hence
$$t^{1-\alpha/n}\lesssim\Phi(t).$$
\end{proof}


\begin{thebibliography}{99}

\bibitem{AMPRR17} {Accomazzo, N. and Mart{\'\i}nez-Perales, J. C. and Rivera-R{\'\i}os, I. P.}, {\it{On {B}loom type estimates for iterated commutators of fractional integrals}},
   {To appear in Indiana Univ. Math. J.}

\bibitem{ACADH17} {Anh, B. and Conde-Alonso, J. M. and Duong, X. T. and Hormozi, M.}, {\it{A note on weighted bounds for singular operators with nonsmooth kernels}},
    {Studia Math.},
 {236},
  {(2017)},
  {245--269}.


\bibitem{AIS01} {Astala, K. and Iwaniec, T. and Saksman, E.}, {\it{Beltrami operators in the plane}}, 
   {Duke Math. J.},
  {107}, {1}, {(2001)}, {27--56}.


\bibitem{BLR11}{Bernardis, A. L and Lorente, M. and Riveros, M. S.}, {\it{Weighted inequalities for fractional integral operators with kernel satisfying {H}{\"o}rmander type conditions}}, 
  {Math. Inequal. Appl}, {14}, {4}, (2011), {881--895}.
	
\bibitem{B93} {Buckley, S. M.}, {\it{Estimates for operator norms on weighted spaces and reverse {J}ensen inequalities}},
   {Trans. Amer. Math. Soc.},
  {340}, {1}, (1993), {253--272}.
		
\bibitem{CLPRR20} {Cejas, M. E. and Li, K. and P{\'e}rez, C. and Rivera-Rios, I.}, {\it{Vector-valued operators, optimal weighted estimates and the $ {C}_p $ condition}},
     {Sci. China Math.}, {(2020)}.		


\bibitem{CU15}  {Cruz-Uribe, D.}, {\it{Elementary proofs of one weight norm inequalities for
              fractional integral operators and commutators}}, {Harmonic analysis, partial differential equations, {B}anach
              spaces, and operator theory. {V}ol. 2}, 5, {Assoc. Women Math. Ser.}, 2017, {183--198}.
      

\bibitem{DPR18} {Dalmasso, E. and Pradolini, G. and Ramos, W.}, {\it{The effect of the smoothness of fractional type operators over their commutators with {L}ipschitz symbols on weighted spaces}},
 {Fract. Calc. Appl. Anal.},   {21},
  {3}, (2018),
  {628--653}.

\bibitem{FH18} {Fackler, S. and Hyt{\"o}nen, T. P.}, {\it{Off-diagonal sharp two-weight estimates for sparse operators}},   {New York J. Math}, {24}, (2018), {21--42}.

\bibitem{GLP08} {Guo, Z. and Li, P. and Peng, L.}, {\it{{$L^p$} boundedness of commutators of {R}iesz transforms associated to {S}chr{\"o}dinger operator}}, {{J}. {M}ath. {A}nal. {A}ppl.}, {341}, {1}, (2008), {421--432}.

\bibitem{G17} {G{\"u}rb{\"u}z, F.}, {\it{Some estimates for generalized commutators of rough fractional maximal and integral operators on generalized weighted {M}orrey spaces}},
  {Canad. {M}ath. {B}ull.}, {60}, {1}, (2017), {131--145}.
	
\bibitem{H12} {Hyt{\"o}nen, T. P.}, {\it{The sharp weighted bound for general {C}alder{\'o}n—{Z}ygmund operators}},
  {Ann. of Math. (2)}, (2012), {1473--1506}.	

\bibitem{IFRR17} {Iba{\~n}ez-Firnkorn, G. H. and Rivera-R{\'\i}os, I. P.},  {\it{Sparse and weighted estimates for generalized {H}{\"o}rmander operators and commutators}}, {Monatsh. {M}ath.}, {191}, {1}, (2020), {125--173}.
	
\bibitem{IFRi18} {Iba{\~n}ez-Firnkorn, G. H. and Riveros, M. S.}, {\it{Certain fractional type operators with {H}{\"o}rmander conditions}}, {Ann. Acad. Sci. Fenn. Math.}, {43}, (2018), {913--929}.
  
\bibitem{K90} {Kurtz, D. S.},  {\it{Sharp function estimates for fractional integrals and related operators}}, {J. Austral. Math. Soc. Ser. A}, {49}, {1}, (1990), {129--137}.	
	
\bibitem{LMPT10} {Lacey, M. T. and Moen, K. and P{\'e}rez, C. and Torres, R. H.},  {\it{Sharp weighted bounds for fractional integral operators}}, {J. {F}unct. {A}nal.}, {259}, {5}, (2010), {1073--1097}.
	
\bibitem{L16} {Lerner, A. K.},  {\it{On pointwise estimates involving sparse operators}}, {New York J. Math}, {22}, (2016), {341--349}.

\bibitem{LN18} {Lerner, A. K. and Nazarov, F.},   {\it{Intuitive dyadic calculus: the basics}}, {Expo. Math.}, (2018).

\bibitem{LORR16} {Lerner, A. K. and Ombrosi, S. and Rivera-R{\'\i}os, I. P.}, {\it{On pointwise and weighted estimates for commutators of {C}alder{\'o}n-{Z}ygmund operators}}, {Adv. Math.}, {319}, (2017), {153--181}.

\bibitem{Li16} {Li, K.}  {\it{Sparse domination theorem for multilinear singular integral operators with ${L}^{r}$-{H}{\"o}rmander condition}}, {Michigan Math. J.}, {67}, {2}, (2018), {253--265}.

\bibitem{LR05} {Lorente, M. and Riveros, M. S.},  {\it{Weights for commutators of the one-sided discrete square function, the {W}eyl fractional integral and other one-sided operators}}, {Proc. Roy. Soc. Edinburgh Sect. A}, {135}, {04}, (2005), {845--862}.

\bibitem{MPTG05}  {Martell, J. and P{\'e}rez, C. and Trujillo-Gonz{\'a}lez, R.}, {\it{Lack of natural weighted estimates for some singular integral operators}}, {Trans. Amer. Math. Soc.}, {357}, {1}, (2005), {385--396}.
  
\bibitem{MW74} {Muckenhoupt, B. and Wheeden, R.}, {\it{Weighted norm inequalities for fractional integrals}},
   {Trans. Amer. Math. Soc.}, {192}, (1974), {261--274}.

\bibitem{P07} {Petermichl, S.},  {\it{The sharp bound for the {H}ilbert transform on weighted {L}ebesgue spaces in terms of the classical characteristic}}, {Amer. J. Math.}, (2007), {1355--1375}.
  
\bibitem{P08} {Petermichl, S.}, {\it{The sharp weighted bound for the {R}iesz transforms}},
   {Proc. Amer. Math. Soc.}, {136}, {4}, (2008), {1237--1249}.

\bibitem{PV02} {Petermichl, S. and Volberg, A.},   {\it{Heating of the {A}hlfors-{B}eurling operator: weakly quasiregular maps on the plane are quasiregular}}, {Duke Math. J.}, {112}, {2}, (2002), {281--305}.

\bibitem{RU14} {Riveros, M. S. and Urciuolo, M.},  {\it{Weighted inequalities for some integral operators with rough kernels}}, {Open Mathematics}, {12}, {4}, (2014), {636--647}.

\bibitem{LibroStein} {Stein, E. M.}, {\it{Singular integrals and differentiability properties of functions}}, {Princeton Mathematical Series, No. 30}, {Princeton University Press, Princeton, N.J.}, 1970.

\end{thebibliography}

\end{document}